\documentclass[letterpaper,11pt,reqno]{amsart}
\usepackage[margin=1.3in]{geometry}
\usepackage{amsmath,amsfonts,amsthm,graphicx,amssymb,latexsym}
\usepackage{mathrsfs}
\usepackage{mathtools}
\usepackage{dsfont}
\usepackage{bbm}
\usepackage{soul}     
\usepackage[utf8]{inputenc}
\usepackage{esint}
\usepackage{caption}
\usepackage{comment}
\usepackage{enumerate}
\usepackage{url}
\usepackage[colorlinks=true, linkcolor=blue, citecolor=ForestGreen]{hyperref}
\usepackage{cite}
\usepackage{cleveref}
\usepackage[usenames, dvipsnames]{color}
\usepackage[usenames, dvipsnames, cmyk]{xcolor}

 \definecolor{labelkey}{rgb}{0,0,1}

\newcommand{\RR}{\mathbb{R}}

\newcommand{\e}{\varepsilon}

\newcommand{\Div}{{\rm div}\,}
\newcommand{\Id}{{\rm Id}\,}
\newcommand{\dist}{{\rm dist}}
\newcommand{\loc}{{\rm loc}}
\newcommand{\diam}{{\rm diam}\,}

\newcommand{\pa}{\partial}

\newcommand{\eps}{\varepsilon}
\newcommand{\xmk}{x_m^k}

\newcommand{\mres}{\mathbin{\vrule height 1.6ex depth 0pt width 
0.13ex\vrule height 0.13ex depth 0pt width 1.3ex}}

\newcommand{\hn}{\mathcal{H}^{n-1}}

\renewcommand{\leq}{\leqslant}
\renewcommand{\geq}{\geqslant}

\newcommand{\xupref}[2]{\hspace{-0.3ex}\stackrel{\eqref{#1}}{#2}} 

\theoremstyle{plain}
\newtheorem{theorem}{Theorem}[section]
\newtheorem{lemma}[theorem]{Lemma}
\newtheorem{corollary}[theorem]{Corollary}
\newtheorem{proposition}[theorem]{Proposition}
\newtheorem*{theorem*}{Theorem}
\newtheorem*{corollary*}{Corollary}

\theoremstyle{definition}
\newtheorem{definition}{Definition}

\newtheorem{remark}[theorem]{Remark}

\newtheorem*{notation*}{Notation}

\numberwithin{equation}{section}
\numberwithin{figure}{section}

\title[Least Wasserstein Distance Between Disjoint Shapes]{Least Wasserstein Distance Between Disjoint Shapes With Perimeter Regularization}

\author{Michael Novack}
\address[Michael Novack]{Department of Mathematics, The University of Texas at Austin, Austin, TX}
\email{michael.novack@austin.utexas.edu}

\author{Ihsan Topaloglu}
\address[Ihsan Topaloglu]{Department of Mathematics and Applied Mathematics, Virginia Commonwealth University, Richmond, VA}
\email{iatopaloglu@vcu.edu}

\author{Raghavendra Venkatraman}
\address[Raghavendra Venkatraman]{Courant Institute of Mathematical Sciences, New York University, New York, NY.}
\email{raghav@cims.nyu.edu}

\date{\today}                                        

\thanks{This is a post-peer-review, pre-copyedit version of an article published in Journal of Functional Analysis. The final
authenticated version is available online at: \url{https://doi.org/10.1016/j.jfa.2022.109732}.}

\begin{document}

\begin{abstract} We prove the existence of global minimizers to the double minimization problem 
\begin{align}\notag
    \inf\Big\{ P(E) + \lambda W_p(\mathcal{L}^n\mres E,\mathcal{L}^n\mres F) \colon |E \cap F| = 0, \, |E| = |F| = 1\Big\},
\end{align}
where $P(E)$ denotes the perimeter of the set $E$, $W_p$ is the $p$-Wasserstein distance between Borel probability measures, and $\lambda > 0$ is arbitrary. The result holds in all space dimensions, for all $p \in [1,\infty),$ and for all positive $\lambda $. This answers a question of Buttazzo, Carlier, and Laborde. 
\end{abstract}

\maketitle

\section{Introduction}\label{sec: intro}
Comparing equal volume \textit{shapes}, i.e., Lebesgue measurable sets $E$ and $F$ in $\mathbb{R}^n$ of equal volume, is a ubiquitous task in numerous applications. From a mathematical standpoint, by identifying shapes with probability measures via their normalized characteristic functions, optimal transportation theory provides natural choices of metrics, the $p$-Wasserstein distances, which metrize the weak convergence of probability measures on compact spaces \cite{Santambrogio}. Indeed, length-minimizing Wasserstein geodesics between equal volume sets, known as displacement interpolants, offer a (length minimizing) path joining the shapes being compared \cite{McCann}.

In this note we investigate the role of perimeter regularization in variational problems involving the Wasserstein distance between equal volume sets. As we subsequently discuss, examples of this type of problem arise in different applications. Our principal goal in this paper is to show that such perimeter-regularized variational problems, even when posed on all of space, \textit{do not suffer a loss of compactness of minimizing sequences}. In order to focus on the technical essence in the simplest possible setting while capturing the main difficulties, we consider the following problem:
\begin{equation}\label{the functional}
    \inf \Big\{ P(E) + \lambda W_p({\mathcal{L}^n \mres} E,{\mathcal{L}^n \mres}F) \colon  E, F \subset \RR^n, \,|E\cap F|=0, \, |E|=|F|=1 \Big\}.
\end{equation}
Here $P(E)$ denotes the perimeter of $E\subset\RR^n$, $W_p$ denotes the $p$-Wasserstein distance {on the space of probability measures}, and $\lambda > 0$ is a constant. { The parameter $p$ belongs to the interval $[1,\infty)$, and} {$\lambda$ represents the strength of the Wasserstein term relative to perimeter.} 

{The problem \eqref{the functional} was recently analyzed by Buttazzo, Carlier and Laborde in \cite{Butt}, in addition to more general minimization problems involving the minimal Wasserstein distance between a measure $\mu$ and measures singular with respect to $\mu$.} In the context of \eqref{the functional}, the authors in \cite{Butt} show, for any $\lambda>0$, the existence of minimizers when admissible sets $E$ and $F$ are required to be subsets of a bounded domain $\Omega$. { In two dimensions, they prove the existence of minimizers for the problem \eqref{the functional} on all of $\mathbb{R}^2$, and conjectured that it should hold in all dimensions.} This {whole-space} result was extended by Xia and Zhou \cite{XZ21} to higher dimensions but under the additional assumptions that $\lambda$ is sufficiently small and that $p<n/(n-2)$. In our result we lift all these restrictions and obtain that minimizers exist in any dimension \emph{and} for all values of $\lambda>0$ and $p \in [1,\infty)$, thereby completely answering the conjecture of Buttazzo, Carlier and Laborde. Precisely, we prove the following theorem:
\medskip

\begin{theorem}[Existence]\label{existence theorem}
For any $\lambda>0$ and $p\in [1,\infty)$, there exists a minimizing pair $(E,F)$ to the problem \eqref{the functional}.
\end{theorem}

\medskip{
The proof of \Cref{existence theorem} is based on tools developed in the context of constrained geometric variational problems on all of space for which symmetrization principles cannot rule out loss of volume at infinity for a minimizing sequence. First, for a minimizing sequence $\{E_m, F_m\}$, the nucleation lemma of Almgren \cite[VI.13]{Alm76} yields a finite number of bounded ``chunks" which contain most of the volume. Then, classical density arguments for constrained perimeter minimizers allow one to argue that the minimizing sequence is essentially confined to finitely many (potentially diverging) balls on which there is no volume loss, at which point lower-semicontinuity of the energy yields the existence of a minimizing pair.

Nonlocal isoperimetric problems are well-studied and consist of minimizing the perimeter functional with some additional nonlocal term that precludes coalescence of sets.} The problem \eqref{the functional} has several interesting mathematical features and exhibits both similarities and differences to other nonlocal isoperimetric models. The behavior of \eqref{the functional} is driven by the competition between the perimeter term and the Wasserstein term. There is an inherent frustration between the two, due the fact that while there exists sequences $\{(E_m, F_m)\}$ of admissible sets to \eqref{the functional} such that $W_p(E_m, F_m)\to 0$, any such sequence necessarily has perimeters approaching infinity, cf. \Cref{vanishing wasserstein implies diverging perimeters}. However, a crucial feature is that the construction of such a sequence can be achieved within a \textit{bounded} set. This is one reason why we are able to prove the existence of minimizers in all parameter regimes, which does not hold for some other examples of perimeter energies perturbed by a nonlocal term. As we recall presently, the celebrated liquid drop model of Gamow displays non-existence phenomena in certain parameter regimes. 

A classical example of a nonlocal isoperimetric problem is the liquid drop model of Gamow (see \cite{Ga1930}),
    \[
        \inf\left\{ P(E) + \int_E \! \int_E |x-y|^{-\alpha} \, dxdy \colon |E|=M \right\},
    \]
where the nonlocal term is given by Riesz-type interactions. Here the two terms present in the energy functional (perimeter and nonlocal interactions) are in direct competition, as in \eqref{the functional}. The surface energy is minimized by a ball whereas the repulsive term prefers to disperse the mass into vanishing components diverging infinitely apart. The parameter of the problem, that is $M$, sets a length scale between these competing forces (see \cite{ChoMurTop} for a review).

There are two major differences between the problem \eqref{the functional} and the aforementioned one:
    \begin{itemize}
        \item The nonlocal isoperimetric problems considered in the literature involve the minimization of functionals over \emph{single} sets of finite perimeter with a volume constraint. The energy functional in \eqref{the functional}, on the other hand, is minimized over a pair of disjoint sets of finite perimeter of equal volume. A similar phenomenon appears in ternary systems (involving both interfacial energy and nonlocal pairwise interactions) with three different phases, where two of which interact via long-range Riesz-type potentials (see \cite{BonKnu}).
        \item Perhaps the most important distinction between \eqref{the functional} and the nonlocal isoperimetric problems studied in the literature is that in our case a minimizing sequence for the repulsive nonlocal term (the Wasserstein distance) does not necessarily consist of vanishing components that are diverging away to infinity (cf. \cite{AlaBroChoTop,KnuMurNov}). Rather, in some sense, it prefers oscillations reminiscent of nonexistence of minimizers in shape optimization problems via nonlocal attractive-repulsive interactions in models of swarming \cite{BurChoTop,FraLie}.
    \end{itemize}
 An important manifestation of these differences is that, as shown by Kn\"upfer and Muratov \cite{KM}, in Gamow's model minimizers \textit{fail to exist} for values of the mass constraint that are larger than a critical value of $M$ and for $\alpha\in(0,2)$ (see also \cite{LO,FKN16} for the physically relevant case of $n=3$, $\alpha=1$, and \cite{FN21} for a newer proof in the general case). This is in striking contrast with our main result for \eqref{the functional}.

\medskip
Let us briefly discuss some of the mathematical literature related to \eqref{the functional}, as it arises in various applications. First, geometric variational problems with a Wasserstein term are useful in the modelling of bilayer membranes. In \cite{PelRog}, Peletier and R\"oger derived the energy
\begin{equation}\label{lipid bilayer model}
     P(E) + \e^{-2} W_1(\mathcal{L}^n \mres E,\mathcal{L}^n \mres F) \quad\textup{for }  E, F \subset \RR^n, \,|E\cap F|=0, \, |E|=|F|=\e ,
\end{equation}
as a simplified model for lipid bilayer membranes. Here the sets $E$ and $F$ represent the densities of the hydrophobic tails and hydrophilic heads, respectively, of the two part lipid molecules. The perimeter term signifies an interfacial energy arising from hydrophobic effects, while the Wasserstein term is a weak remainder of the bonding between the head and tail particles. The authors in \cite{LusPelRog,PelRog} considered the asymptotic expansion as $\e\to 0$ of the energy in $\mathbb{R}^2$ and $\mathbb{R}^3$ and identified a limiting energy concentrated on a codimension one set. The competition described earlier between the two terms in the energy drives the system toward \textit{partially localized} structures that are thin in one direction ($\sim\e$) and extended in the remaining directions. Since \eqref{lipid bilayer model} is equivalent to \eqref{the functional} up to rescaling and choosing the correct $\lambda=\lambda(\e)$, our existence theorem applies to \eqref{lipid bilayer model}. Nonlocal isoperimetric problems (mostly related to models of diblock copolymers) where the perimeter functional is perturbed by a nonlocal term involving the 2-Wasserstein distance have also appeared elsewhere in the literature (cf. \cite{BouPelRop,PelVen}).

In a completely different line of research, in the recent article \cite{LPS}, Liu, Pego, and Slep\v{c}ev study incompressible flows between equal volume shapes, as critical points for action, given by kinetic energy along transport paths that are constrained to be characteristic function densities. Formally, viewing the space of equal volume shapes as an infinite dimensional ``manifold'', the  critical points for action are geodesics -- they verify incompressible Euler equations for an inviscid potential flow with zero pressure, and \textit{zero surface tension along free boundaries}. The authors in \cite{LPS} find that, in particular, locally minimizing action exhibits an instability associated with microdroplet formation. They show that any two shapes of equal volume can be approximately connected by what they refer to as an ``Euler spray'', a countable superposition of ellipsoidal geodesics. Furthermore, associated with the aforementioned instability, the infimum of action, which is equal to the squared $2$-Wasserstein distance, is not attained. 

Unlike \cite{LPS}, we do not focus on \textit{paths} joining shapes -- investigating the role of surface tension in alleviating the microdroplet instability alluded to above is an interesting research direction that we hope to pursue elsewhere. For now we simply note that in the absence of the perimeter regularization in \eqref{the functional}, minimizing sequences disintegrate into tiny ``microdroplets'', driving the minimum energy to zero, a form of microdroplet instability (see \Cref{vanishing wasserstein implies diverging perimeters}). We believe our technical contributions precluding the loss of compactness via microdroplet formation will be useful in studying the effect of including surface tension in \cite{LPS}.

Finally, we mention some future directions and questions that remain regarding \eqref{the functional}. While the one-dimensional calculations in \cite[Example 4.4]{Butt} determine the minimizers depending on $\lambda$ explicitly, the characterization of minimizers for any $\lambda>0$ in higher dimensions remains an open problem. 

\medskip

Shortly after submission of the present article, Candau-Tilh and Goldman uploaded a preprint on arXiv which studies the same minimization problem (see \cite{C-TG22}). They obtain the existence of minimizers via an alternative argument. They also characterize global minimizers in the small $\lambda$ regime, partially answering a question left open in our paper.

\section{Notation and Preliminaries}\label{sec: preliminaries}
We introduce some notation that we will use throughout the paper. Let $B(x,r)$ denote the open ball in $\RR^n$ centered at $x$ with radius $r$, and let $\omega_n := |B(0,1)|$. For any Lebesgue measurable set $E \subset \RR^n$, $|E|$ is the Lebesgue measure of $E$. Finally, we use uppercase $C_n$, $C_p$, and $C_{n,p}$ to refer to constants that depend on one or both of the spatial dimension $n$ and $p\in [1, \infty)$. The values of these constants may change from line to line. An exception to this convention is \Cref{nucleation lemma}, so we denote the dimensional constant appearing there by the lowercase $c(n)$.

We work within the setting of sets of finite perimeter in $\mathbb{R}^n$ (see e.g. \cite{Maggi}). Given a Lebesgue measurable set $E \subset \RR^n$ we use the perimeter functional in the sense of De Giorgi, defined by 
\begin{align*}
    P(E) := \sup\left\{ \int_{E} \mathrm{div}\, V(x)\,dx \colon V \in C_c^1(\RR^n;\mathbb{R}^n), |V| \leqslant 1 \right\}.
\end{align*}
This notion of perimeter possesses properties such as lower-semicontinuity under $L^1_\loc$-convergence, which is immediate from the definition, and compactness. 

In the sequel, we will need the following nucleation lemma, due to Almgren \cite[VI.13]{Alm76} and quoted from \cite[Lemma 29.10]{Maggi}.
\begin{lemma}[Nucleation]\label{nucleation lemma}
For every $n\geq 2$, there exists a positive constant $c(n)$ with the following property. If $E$ is of finite perimeter, $0 < |E| < \infty$, and
\begin{equation}\notag
    \eps \leq \min \left\{|E|, \frac{P(E)}{2nc(n)} \right\},
\end{equation}
then there exists a finite family of points $x_i \subset \mathbb{R}^n$, $1\leq i \leq I$ such that
\begin{equation}\notag
    \left|E \setminus \bigcup_{1\leq i\leq I} B(x_i,2) \right| < \eps, 
\end{equation}
\begin{equation}\label{lower bound on i ball}
    \left|E\cap B(x_i,1) \right| \geq \left(c(n) \frac{\eps}{P(E)} \right)^n.
\end{equation}
Moreover, $|x_i-x_{i'}|>2$ for every $i\neq i'$, and
\begin{equation}\label{i bound}
      I < |E| \left(\frac{P(E)}{c(n)\eps} \right)^n.
\end{equation}
\end{lemma}

\medskip

\begin{remark}[Nucleation/compactness]\label{nucleation/compactness remark}
We will often employ the nucleation lemma, in particular the conclusion \eqref{lower bound on i ball}, in conjunction with the compactness theorem for sets of finite perimeter (cf. for example \cite[Corollary 12.27]{Maggi}) to obtain a positive measure subsequential $L^1_\loc$-limit of a suitable sequence $\{E_m\}$. Precisely, if $\{E_m\}$ is a sequence of sets of finite perimeter satisfying
\begin{equation}\notag 
    0<\varepsilon :=\inf_{m} \, \min\left\{ |E_m|, \frac{P(E_m)}{2nc(n)}\right\},\qquad \sup_{m} P(E_m) < \infty,
\end{equation}
then, up to extraction of a non-relabeled subsequence, there exists a non-empty set $E$ and sequence $\{x_m\} \subset \mathbb{R}^n$ such that $(E_m-x_m) \overset{\loc}{\to} E$ and
\begin{equation}\label{volume of E on B1}
    \left|E\cap B(0,1) \right| \geq \left(c(n) \frac{\eps}{\sup P(E_m)} \right)^n>0.
\end{equation}
Here the local convergence for sets is the strong $L^1_\loc$ convergence of the corresponding characteristic functions. We remark that this compactness property has also been obtained by Frank and Lieb in \cite{FL15} using different arguments.  
\end{remark}

\bigskip

The next lemma is an amalgamation of several standard arguments \cite[Lemmas 17.21 and 17.9]{Maggi}. It allows for comparison of the energies of a minimizing sequence against local variations which do not necessarily preserve the volume constraint, cf. \eqref{almost perimeter minimality of E_m}, which is useful in the derivation of density estimates for example. For convenience we include the proof of this lemma in the appendix.

\begin{lemma}[Volume-fixing variations along a sequence]\label{volume fixing lemma}
Let $E$ be a set of finite perimeter and $A$ be an open set such that $\hn(\pa^\ast E \cap A)>0$. Suppose also that $\{E_m\}$ satisfy
\begin{equation}\notag
    \sup P(E_m;A) \leq M < \infty
\end{equation}
and $E_m \overset{\loc}{\to }E$ in $\mathbb{R}^n$. Then there exists $\sigma_0 = \sigma_0(E,A,M)>0$ and $C_0=C_0(E,A,M)<\infty$ such that for every $\sigma \in (-\sigma_0, \sigma_0)$ and large enough $m$ there exist sets of finite perimeter $G_m$ with $G_m \triangle E_m\subset\!\subset A$ and
\begin{gather}\label{volume fixing equation}
    |G_m \cap A| = |E_m\cap A| + \sigma, \\ \label{L1 and perimeter bounds}
    |G_m \triangle E_m| \leq C_0|\sigma|, \quad\text{and}\quad|P(G_m;A)- P(E_m;A)| \leq C_0|\sigma|.
\end{gather}
\end{lemma}

\bigskip

We turn to recalling notions from optimal transport that we use throughout the paper (see \cite{Ambrosio,Santambrogio, Villanitopics,Villanioldandnew} for further details). The family of finite, positive Borel measures on $\mathbb{R}^n$ is denoted by $\mathcal{M}_+(\mathbb{R}^n)$. We work with this class instead of the usual space of Borel probability measures since it will sometimes be useful to have a notion of transport between measures of equal mass other than 1; this of course entails no significant change in the theory. Given $\mu$, $\nu\in \mathcal{M}_+(\mathbb{R}^n)$ with $\mu(\RR^n) = \nu(\RR^n)$, we let $\Pi(\mu,\nu)$ be the set of all couplings between $\mu$ and $\nu$:
\begin{align*}
    \Pi(\mu,\nu) := \left\{ \gamma \in \mathcal{M}_+(\RR^n \times \RR^n) \colon (\pi_1)_\# \gamma =\mu, \, (\pi_2)_\# \gamma = \nu \right\},
\end{align*}
where $\#$ is the push-forward operation, and $\pi_1$, $\pi_2$ respectively denote projections onto the first and second copies of $\RR^n$. A transport map from $\mu$ to $\nu$ is a map $T\colon\RR^n \to \RR^n$ such that $T_\# \mu = \nu$. Any such $T$ induces a coupling $\gamma$ via the relation $\gamma = (\Id\times T)_\# \mu$. When $\mu= \mathcal{L}^n\mres E$ and $\nu=\mathcal{L}^n\mres F$, we will refer to $T$ as transporting $E$ to $F$.

Kantorovich's problem with cost $c(x,y)$ for measures $\mu$, $\nu$ of equal total mass is
\begin{equation}\notag
\textbf{K}_c(\mu,\nu)=\inf\left\{ \int_{\RR^n \times \RR^n} c(x,y) \,d\gamma(x,y) \colon \gamma \in \Pi(\mu,\nu)\right\}.
\end{equation}
Since we are interested in the case where $\mu$ and $\nu$ are the restrictions of Lebesgue measure to two subsets of $\mathbb{R}^n$ and the cost is
\begin{equation}\notag
c_p(x,y) := |x-y|^p,
\end{equation}
the existence of a solution to Kantorovich's problem in this instance is relevant. For stronger versions of this theorem and more comprehensive discussions of the vast mathematical literature on optimal transport, we refer the reader to the monographs mentioned above and the references therein.

\begin{theorem}[Existence of an optimal transport map]\label{existence of optimal transport maps}
Let $p\in [1,\infty)$, and suppose $E$ and $F$ are Lebesgue measurable sets with $|E|=|F|>0$. Then there exists a map $\Phi \colon \mathbb{R}^n\to \mathbb{R}^n$, called an optimal transport map, such that $\Phi_\# (\mathcal{L}^n\mres E) = \mathcal{L}^n\mres F$ and
\begin{equation}\notag
     \int_{E} |x-\Phi(x) |^p \, dx =  \textup{\textbf{K}}_{c_p}(\mathcal{L}^n\mres E,\mathcal{L}^n\mres F).
\end{equation}
\end{theorem}

\medskip

{Using optimal transport theory, one may define a distance between finite Lebesgue measure sets. This notion and more general ones involving mutually singular measures were analyzed in \cite{Butt}.}
\begin{definition}\label{definition of script W}
For positive Lebesgue measure sets $E$ and $F$ with equal measure, let
\begin{equation}\label{wasserstein distance between two sets}
    W_p(E,F) := \textbf{K}_{c_p}(\mathcal{L}^n\mres E,\mathcal{L}^n\mres F)^{\frac{1}{p}}.
\end{equation}
Also, we set
\begin{equation}\notag
   \mathcal{W}_p(E):= \inf\{W_p(E,F) \colon |F|=|E|, \, |F \cap E|=0 \},
\end{equation}
with the convention that $\mathcal{W}_p(E)=0$ if $|E|=0$.
\end{definition}

\medskip

When $\mathcal{L}^n\mres E$, $\mathcal{L}^n\mres F$ are in the space of Borel probability measures with finite $p$\textsuperscript{th} moments $ \mathcal{P}_p(\mathbb{R}^n)$, the definition \eqref{wasserstein distance between two sets} coincides with the much-studied $p$-Wasserstein distance between two disjoint sets, hence the duplicate notation. The rest of the preliminaries are dedicated to the properties of $\mathcal{W}_p$ necessary for our analysis.

\begin{lemma}[Properties of $\mathcal{W}_p$]\label{basic properties lemma}
Let $E\subset \mathbb{R}^n$ be Lebesgue measurable.
\begin{enumerate}[(i)]
\item \textup{(Monotonicity)} If $E\subset F$, where $F$ is Lebesgue measurable, then $\mathcal{W}_p(E) \leq \mathcal{W}_p(F)$.
\item \textup{(Positivity)} If $|E|>0$, then $\mathcal{W}_p(E) >0$.
\item \textup{(Scaling)} For any $r\geq 0$,
\begin{equation}\label{scaling of script W}
    \mathcal{W}_p(rE) = r^{1+\frac{n}{p}}\mathcal{W}_p(E).
\end{equation}
\item \textup{($L^q$-bound)} There exists $C_n$ such that
\begin{equation}\notag
\mathcal{W}_p(E) \leq C_n|E|^{\frac{1}{p}+\frac{1}{n}};
\end{equation}
(cf. \cite[Equation 4.2]{XZ21} for the same statement when $E$ is bounded).
\end{enumerate}
\end{lemma}

\begin{proof}
Items $(i)$ and $(iii)$ follow immediately from the definition of $\mathcal{W}_p$. By $(i)$, it suffices to prove $(ii)$ in the case that $|E|>0$ and $E$ is bounded. 

Suppose for a contradiction that $F_m$ is such that $\mathbf{K}_{c_p}(\mathcal{L}^n\mres E, \mathcal{L}^n\mres F_m)\to 0$. In this case, $E$ and $F_m$ have finite $p$\textsuperscript{th} moments, so by the properties of the $p$-Wasserstein distance, $W_p(E, F_m)\to 0$ implies that $\mathcal{L}^n\mres F_m \overset{\ast}{\rightharpoonup} \mathcal{L}^n\mres E$ {(see for example \cite[Theorem 5.11]{Santambrogio})}. But this is incompatible with $|E\cap F_m|=0$ and $|E|=|F_m|>0$, so  we have a contradiction.

For $(iv)$, by the scaling \eqref{scaling of script W}, it is enough to prove the claim when $|E|=1$. Divide $\mathbb{R}^n$ into disjoint cubes $Q_j$ of volume 2. For each $j$, since $|Q_j|=2$ and $|E\cap Q_j|\leq 1$, we can find $F_j\subset Q_j$ such that $|F_j|=|Q_j \cap E|$ and $|F_j\cap E|=0$. Let $T_j$ transport $E \cap Q_j$ onto $F_j$, and set $F=\bigcup_j F_j$. Then it is easy to see that the map $T$ defined by
\begin{equation*} 
T(x) = T_j(x) \qquad \text{for }x\in E\cap Q_j
\end{equation*}
transports $E$ onto $F$ and satisfies $|x - T(x)| \leq \diam (Q_j)$ for $x \in   E \cap Q_j$. Thus
\begin{equation}\notag
    \mathcal{W}_p(E) \leq \left(\int_E |x - T(x)|^p\,dx \right)^{\frac{1}{p}} \leq  C_n,
\end{equation}
since $\mathrm{diam}(Q_j) =: C_n,$ independent of $j,$ and $|E| = 1.$ The claim follows.
\end{proof}

\begin{proposition}[Continuity of $\mathcal{W}_p$ with respect to $L^1$-convergence]\label{L1 control of script W}
There exists $C_{n,p}$ such that for any $|E|$, $|\tilde{E}|$,
\begin{equation}\label{L1 control to the p}
    \left|\mathcal{W}_p^p(\tilde{E}) - \mathcal{W}_p^p(E) \right|\leq C_{n,p} \max \{|E|^{\frac{p}{n}},|\tilde{E}|^{\frac{p}{n}} \}|E\triangle \tilde{E}|
\end{equation}
and
\begin{equation}\label{L1 control}
    \left|\mathcal{W}_p(\tilde{E}) - \mathcal{W}_p(E) \right| \leq C_{n,p} \max\{\mathcal{W}_p^{1-p}(E),\mathcal{W}_p^{1-p}(\tilde{E}) \}\max \{|E|^{\frac{p}{n}},|\tilde{E}|^{\frac{p}{n}} \}|E\triangle \tilde{E}|.
\end{equation}
\end{proposition}

\medskip

\begin{remark}\notag
When $E$ and $F$ are both bounded with unit measure, \Cref{L1 control of script W} is contained in \cite[Lemma 4.5]{Butt}.
\end{remark}

\medskip

\begin{proof}[Proof of \Cref{L1 control of script W}]
First we demonstrate how \eqref{L1 control} follows from \eqref{L1 control to the p}. By applying the mean value theorem to the function $t \mapsto t^{1/p}$, we deduce that
\begin{equation}\notag
\left|\mathcal{W}_p(\tilde{E}) - \mathcal{W}_p(E) \right|\leq \frac{1}{p}\max\{\mathcal{W}_p^{1-p}(E),\mathcal{W}_p^{1-p}(\tilde{E}) \} \left|\mathcal{W}_p^p(\tilde{E}) - \mathcal{W}_p^p(E) \right|.
\end{equation}
The bound \eqref{L1 control} follows immediately from this equation and \eqref{L1 control to the p}.\par

It remains to prove \eqref{L1 control to the p}. Without loss of generality, 
\begin{equation}\label{greater transport cost}
\mathcal{W}_p^p(\tilde{E})\geq \mathcal{W}_p^p(E).
\end{equation}
We may also assume that
\begin{equation}\label{unit volume assumption}
|\tilde{E}|=1;
\end{equation}
the full case then follows from rescaling. Fix any $F$ with $|F|=|E|$ and $|F\cap E|=0$. If we can show that 
\begin{equation}\label{L1 control of transport maps}
    \mathcal{W}_p^p (\tilde{E}) - W_p^p(E,F) \leq C_{n,p}\max \{|E|^{\frac{p}{n}},1 \} |E\triangle \tilde{E}|,
\end{equation}
then, in light of \eqref{greater transport cost}, taking the infimum over $F$ disjoint from $E$ gives \eqref{L1 control to the p} when $|\tilde{E}|=1$.\par

To show \eqref{L1 control of transport maps} under the assumptions \eqref{greater transport cost} and \eqref{unit volume assumption}, first consider the case that
$$|E|\geq 2.$$
Then since $|\tilde{E}|=1$, we have $|E \triangle \tilde{E}| \geq 1$, so that 
\begin{equation}\label{lower bound for large E}
    \max\{ |E|^{\frac{p}{n}}, 1\} |E \triangle \tilde{E}| \geq 2^{\frac{p}{n}}.
\end{equation}
In addition,
\begin{equation}\notag
    \mathcal{W}_p^p (\tilde{E}) - W_p^p(E,F) \leq \mathcal{W}_p^p(\tilde{E}) \leq C_{n,p},
\end{equation}
which together with \eqref{lower bound for large E} gives \eqref{L1 control of transport maps} after suitably modifying $C_{n,p}$. For the rest of the proof of \eqref{L1 control of transport maps}, we therefore assume that
\begin{equation}\label{small E assumption}
    |E|\leq 2.
\end{equation}
Let $\Phi$ be an optimal transport map from $E $ to $F $, which exists by \Cref{existence of optimal transport maps}. The idea is to modify $\Phi$ to create a transport map $\tilde{\Phi}$ for $ \tilde{E}$ (to a set of the appropriate measure), which allows for comparison between $\mathcal{W}_p^p(\tilde{E})$ and $W_p^p(E,F)$. When $x\in E \cap \tilde{E}$ and $\Phi(x) \notin \tilde{E}$, we can define $\tilde{\Phi}$ simply by using $\Phi$:
\begin{equation}\label{transport of the good points}
    \tilde{\Phi}(x) =
      \Phi(x) \quad \textup{if }x \in \tilde{E} \cap E \cap \Phi^{-1}(F \cap \tilde{E}^c).
\end{equation}
For the rest of the points in $\tilde{E}$, we must make a new definition. We partition $\mathbb{R}^n$ into cubes $Q_j$ of volume 4. Since $|Q_j\setminus (\tilde{E}\cup F)|\geq 1$ and $|\tilde{E}|= 1$, there exist measurable sets $D_j \subset Q_j$ such that
\begin{equation}\notag
   D_j \cap \tilde{E} = \emptyset = D_j \cap F \quad\textup{and}\quad|D_j| = |Q_j \cap \tilde{E} \cap (E^c \cup \Phi^{-1}(F \cap \tilde{E}))| \leq 1.
\end{equation}
We may obtain optimal transport maps $\Phi_j$ from $Q_j \cap \tilde{E} \cap (E^c \cup \Phi^{-1}(F \cap \tilde{E}))$ to $D_j$ and define
\begin{equation}\notag
    \tilde{\Phi}(x) =
      \Phi_j(x) \quad\textup{if }x\in Q_j \cap \tilde{E} \cap (E^c \cup \Phi^{-1}(F \cap \tilde{E})).
  \end{equation}
Before estimating the energy difference, we note that since $\Phi$ is a transport map and $F\subset E^c$,
\begin{equation}\label{measure preserving estimate}
    |\tilde{E} \cap  \Phi^{-1}(F \cap \tilde{E})| \leq |\Phi^{-1}(\tilde{E} \cap F)|=|\tilde{E}\cap F|\leq |\tilde{E} \cap E^c|.
\end{equation}
Then
\begin{align}\notag
\mathcal{W}_p^p(\tilde{E}) - W_p^p(E,F)
    &\leq\int_{\tilde{E}}|x - \tilde{\Phi}(x)|^p\,dx - \int_{E}|x - \Phi(x)|^p\,dx  \notag \\ \notag
& \leq\sum_j\int_{Q_j \cap \tilde{E} \cap (E^c \cup \Phi^{-1}(F \cap \tilde{E}))}|x - \Phi_j(x)|^p\,dx   \\ \notag
&\leq \sum_j \diam (Q_j)^p (|Q_j \cap \tilde{E} \cap E^c|+|Q_j \cap \tilde{E}\cap \Phi^{-1}(F \cap \tilde{E})|)  \\ \notag
&= \diam (Q_j)^p (|\tilde{E}\cap E^c|+ |\tilde{E} \cap  \Phi^{-1}(F \cap \tilde{E})| )\\ \label{bound for small E}
&\hspace{-.24cm} \xupref{measure preserving estimate}{\leq} 2\,\diam(Q_j)^p |\tilde{E}\triangle E|.
\end{align}
Since $1\leq \max \{|E|^{\frac{p}{n}},1 \}\leq 2^{\frac{p}{n}}$, \eqref{bound for small E} implies \eqref{L1 control of transport maps}. The proof is complete.
\end{proof}

\medskip

\begin{lemma}[Non-existence of minimizers for $\mathcal{W}_p$]\label{vanishing wasserstein implies diverging perimeters}
There exists a sequence $\{ (E_m, F_m) \}$ such that $|E_m\cap F_m|=0$, $|E_m|=|F_m|=1$, and
\begin{equation}\notag
    W_p(E_m,F_m) \to 0.
\end{equation}
Furthermore, for any sequence satisfying those three properties,
\begin{equation}\label{diverging perimeters}
    P(E_m),\, P(F_m) \to \infty.
\end{equation}
\end{lemma}
\begin{proof}
We omit a full proof of the construction of such a sequence, which is straightforward. There are many ansatzes one could use; for example, $E_m$ could be a single thin, arbitrarily long cylinder, and $F_m$ a suitable tubular neighorhood. Alternatively, $E_m$ and $F_m$ could be suitably many disjoint arbitrarily small balls and corresponding annuli around them. The latter example may be viewed as an analogue of the microdroplet instability discovered in \cite{LPS} in our static setting. 

To prove \eqref{diverging perimeters}, assume for contradiction that $W_p(E_m,F_m)\to 0$ but $\limsup P(E_m)< \infty$. By \Cref{nucleation/compactness remark}, the uniform perimeter bound implies that, up to translations which we ignore, there exists a set $E$ with $|E\cap B(0,1)| >0$ and $E_m \overset{\loc}{\to} E$. Therefore, $E_m \cap B(0,1) \to E \cap B(0,1)$, and so by the $L^1$-continuity of $\mathcal{W}_p$, \begin{equation}\notag
0<\mathcal{W}_p(E \cap B(0,1))   = \liminf_{m \to \infty} \mathcal{W}_p(E_m \cap B(0,1))   \leq \liminf_{m\to \infty} \mathcal{W}_p(E_m, F_m) =0.
\end{equation}
We have thus arrived at a contradiction. The proof that $P(F_m)$ diverges is the same.
\end{proof}

\begin{remark} In their paper \cite{C-TG22}, Candau-Tilh and Goldman obtain the following interpolation inequality
\begin{equation}\label{interpolation inequality}
    W_p(E,F) P(E) \geq C(n) |E|^{1+\frac{1}{p}}.
\end{equation}
As a consequence of this inequality one can obtain \eqref{diverging perimeters} in Lemma \ref{vanishing wasserstein implies diverging perimeters}. Here we provide an alternative proof of this interpolation inequality which effectively quantifies our proof of \eqref{diverging perimeters}.
\end{remark}
\begin{proof}[Proof of \eqref{interpolation inequality}]
We first observe that there exists $C_{n}>0$ such that if $|E\cap B_r|\geq 3\omega_n r^n/4$, then
\begin{equation}\label{large wass}
\mathcal{W}_p(E \cap B_r) \geq C_{n} r^{1+\frac{n}{p}}.
\end{equation}
This is due to the fact that at least $\omega_n r^n/4$ of the mass of $E \cap B_r$ is contained in $B_{3r/4}$ and must be transported outside $B_r$. Let $|E|=|F|$ and $|E\cap F|=0$ and consider any Lebesgue point $x\in E^{(1)}$. By the continuity of $r\to |E\cap B_r(x)|$, the fact that $x\in E^{(1)}$, and the intermediate value theorem, there exists
\begin{equation}\label{radius bound for Besicovitch}
r_x \leq \left(\frac{4|E|}{3\omega_n}\right)^{\frac{1}{n}}
\end{equation}
such that $|E\cap B_{r_x}(x)|=3\omega_n r_x^n/4 $. By \eqref{radius bound for Besicovitch}, we can apply the Besicovitch covering theorem to the family of closed balls $\mathcal{F}=\{\overline{B}_{r_x}:x\in E^{(1)} \}$, to obtain subfamilies $\mathcal{F}_1$, $\dots$, $\mathcal{F}_{\xi(n)}$, each of which consists of disjoint balls, such that
\begin{equation}\notag
    E^{(1)}\subset \bigcup_{i=1}^{\xi(n)}\bigcup_{\overline{B}_{r_x}\in \mathcal{F}_i}\overline{B}_{r_x}\,.
\end{equation}
By the relative isoperimetric inequality, since $|E\cap B_{r_x}|=3\omega_n r_x^n/4$, we have for some $c_n$
\begin{equation}\label{relative iso}
    P(E; B_{r_x}) \geq c_n r_x^{n-1}\quad\forall x\in E^{(1)} .
\end{equation}
Also, with $\Phi$ denoting the optimal transport map from $E$ to $F$, we may use the observation \eqref{large wass} to see that
\begin{align}\label{wass scale}
    W_p(E\cap B_{r_x}, \Phi(E\cap B_{r_x})) \geq \mathcal{W}_p(E\cap B_{r_x}(x)) \geq C_{n}r^{1+\frac{n}{p}}\,.
\end{align}
Finally, combining \eqref{relative iso}-\eqref{wass scale} with H\"{o}lder's inequality, we may estimate
\begin{align*}
    W_p(E,F) P(E) &\geq \xi(n)^{-1-\frac{1}{p}}\left(\sum_{i=1}^{\xi(n)}\sum_{\overline{B}_{r_x}\in \mathcal{F}_i}\int_{E\cap \overline{B}_{r_x}}|z-\Phi(z)|^p\,dz\right)^{\frac{1}{p}}\left(\sum_{i=1}^{\xi(n)}\sum_{\overline{B}_{r_x}\in \mathcal{F}_i}P(E;\overline{B}_{r_x})\right) \\
    &\geq \xi(n)^{-2}\left[\left(\sum_{i=1}^{\xi(n)}\sum_{\overline{B}_{r_x}\in \mathcal{F}_i}C_{n}^pr_x^{p+n}\right)^{\frac{1}{p+1}}\left(\sum_{i=1}^{\xi(n)}\sum_{\overline{B}_{r_x}\in \mathcal{F}_i}c_{n}r_x^{n-1}\right)^{\frac{p}{p+1}}\right]^{\frac{p+1}{p}} \\
    &\geq \frac{C_nc_n}{\xi(n)^{2}}\left(\sum_{i=1}^{\xi(n)}\sum_{\overline{B}_{r_x}\in \mathcal{F}_i}r_x^{\frac{p+n}{p+1}}r_x^{\frac{p(n-1)}{p+1}}\right)^{\frac{p+1}{p}} \\
    &= \frac{C_nc_n}{\xi(n)^{2}}\left(\sum_{i=1}^{\xi(n)}\sum_{\overline{B}_{r_x}\in \mathcal{F}_i}r_x^n\right)^{\frac{p+1}{p}} \\
    &= \frac{C_nc_n}{\xi(n)^{2}}\left(\sum_{i=1}^{\xi(n)}\sum_{\overline{B}_{r_x}\in \mathcal{F}_i}\frac{4|E\cap B_{r_x}(x)|}{3\omega_n}\right)^{\frac{p+1}{p}} \\
    &\geq \tilde{C}_n|E|^{1+\frac{1}{p}},
\end{align*}
where in the last equality we have used the fact that $|E\cap B_{r_x}| = 3\omega_n r_x^n/4$.
\end{proof}

\bigskip

The last preliminary result is drawn from \cite{Butt} and \cite{XZ21}.
\begin{theorem}\label{Buttazzo existence theorem}
Let $E$ be a bounded Lebesgue measurable set.
\begin{enumerate}[(i)]
    \item \textup{\cite[Theorem 3.21]{Butt}} There exists a Lebesgue measurable set $F$ with $|F|=|E|$ and $|E\cap F|=0$ and an optimal transport map $\Phi$ from $E$ to $F$ such that
\begin{equation}\notag
    W_p(E,F) = \mathcal{W}_p(E).
\end{equation}
\item \textup{\cite[Lemma 4.3]{XZ21}} There exists $C_n$ such that for $\mathcal{L}^n$-a.e. $x\in E$,
\begin{equation}\notag
    |x - \Phi(x)| \leq C_n |E|^{\frac{1}{n}}.
\end{equation}
\end{enumerate}
\end{theorem}

\medskip

\begin{remark}[Additivity of $\mathcal{W}_p^p$]\label{additivity of script W}
Arguing directly from items $(i)$ and $(ii)$ of the above theorem, it follows that if $E_1, \dots, E_K$ are bounded sets such that 
\begin{equation}\notag
\dist(E_k,E_{k'})\geq 2C_n \max_{1\leq j\leq K} |E_j|^{\frac{1}{n}} \quad \textup{for }k\neq k',
\end{equation}
then the sets $F_k$ minimizing $W_p(E_k,F_k)$ are pairwise disjoint and
\begin{equation}\notag
    W_p^p\left(\bigcup_k E_k,\bigcup_k F_k\right) = \mathcal{W}^p_p\left(\bigcup_{k=1}^K E_k\right) = \sum_{k=1}^K \mathcal{W}^p_p (E_k)= \sum_{k=1}^K W^p_p(E_k,F_k) .
\end{equation}
\end{remark}

\bigskip
\section{Proof of \Cref{existence theorem}}\label{sec: proof}

We write the main functional as
\begin{equation}\label{the unconstrained functional}
   \mathcal{G}(E):= P(E) + \lambda \mathcal{W}_p(E),
\end{equation}
where $\mathcal{W}_p$ is given as in \Cref{definition of script W}.

\medskip

\begin{proof}[Proof of \Cref{existence theorem}] We prove this theorem in multiple steps.

\medskip

\noindent\textit{Step one:} First, we extract a nontrivial set $E^0$ which is the limit of sets $E_m$ corresponding to a minimizing sequence $\{E_m\}_m$ with
\begin{equation}\label{almost minimality}
    P(E_m) + \lambda \mathcal{W}_p (E_m) \leq \inf \mathcal{G} + \frac{1}{m}.
\end{equation}
From this inequality we have the immediate upper bound
\begin{equation}\label{perimeter upper bound}
    P(E_m) \leq 1+ \inf\mathcal{G}.
\end{equation}
on the perimeters. Since in addition $|E_m|=1$ for all $m$, we may then apply the nucleation lemma and compactness as in \Cref{nucleation/compactness remark}. Therefore, up to a subsequence which we do not relabel and translations which, without loss of generality, are trivial, there exists a set $E^0$ with
\begin{gather}\notag
   0  \xupref{volume of E on B1}{<} |E^0|  \leq 1, \\ \notag
    E_m \overset{\loc}{\to}\, E^0 \quad\textup{in }\mathbb{R}^n.
\end{gather}

\medskip

\noindent \textit{Step two:} Here we identify $\delta$, $\alpha >0$ such that if $\tilde{E}_m$ is any set with $|E_m \triangle \tilde{E}_m|\leq \delta$, then 
\begin{equation}\label{symmetric difference inequality}
     \left|\mathcal{W}_p(E_m) - \mathcal{W}_p(\tilde{E}_m) \right| \leq C_{n,p} \alpha^{1-p}  |E_m \triangle \tilde{E}_m|.
\end{equation}
We first observe that by the uniform perimeter bound and \Cref{vanishing wasserstein implies diverging perimeters} we can consider
\begin{equation}\label{wasserstein lower bound}
    \alpha := \inf_m \mathcal{W}_p(E_m) >0.
\end{equation}
Due to the continuity of $\mathcal{W}_p$ with respect to $L^1$-convergence from \eqref{L1 control to the p}, we may choose $0<\delta\leq 1$ small enough so that if $|E_m \triangle \tilde{E}_m| \leq \delta$,
\begin{align}\notag
    \mathcal{W}_p^p(\tilde{E}_m) &\geq \mathcal{W}_p^p(E_m) - C_{n,p}\max \{|E_m|^{\frac{p}{n}},|\tilde{E}_m|^{\frac{p}{n}} \} |E_m \triangle \tilde{E}_m| \\ \notag
    &\geq \alpha^p - C_{n,p}(1+ \delta)^{\frac{p}{n}}\delta\\ \notag
    &\geq \frac{\alpha^p}{2^p}.
\end{align}
Then since $\mathcal{W}_p(E_m)$ and $\mathcal{W}_p(\tilde{E}_m)$ are both bounded from below by $\alpha/2$, \eqref{L1 control} gives
\begin{align}\notag
    \left|\mathcal{W}_p(\tilde{E}_m) - \mathcal{W}_p(E_m) \right| &\leq C_{n,p} \max\{\mathcal{W}_p^{1-p}(E_m),\mathcal{W}_p^{1-p}(\tilde{E}_m) \}\max \{|E_m|^{\frac{p}{n}},|\tilde{E_m}|^{\frac{p}{n}} \}|E_m\triangle \tilde{E}_m| \\ \notag
    &\leq C_{n,p}\alpha^{1-p}2^{p-1}(1+\delta)^{\frac{p}{n}}|E_m \triangle \tilde{E}_m|.
\end{align}
Upon recalling that $\delta\leq 1$ and modifying $C_{n,p}$, we have shown \eqref{symmetric difference inequality}.

\medskip

\noindent \textit{Step three:} In this step, we utilize \eqref{symmetric difference inequality} and \Cref{volume fixing lemma}, the volume-fixing variations lemma, to show that there exists $r_0$ and $\Lambda>0$ such that for all $m$ large enough, $E_m$ satisfies the inequality
\begin{equation}\label{almost perimeter minimality of E_m}
    P(E_m) \leq P(\tilde{E}_m) + \Lambda|E_m \triangle \tilde{E}_m| + \frac{1}{m}\qquad \textup{if } E_m \triangle \tilde{E}_m \subset\!\subset B(x,r),\ 0<r<r_0.
\end{equation}

Fix $x$ and consider $\tilde{E}_m$ with $E_m \triangle \tilde{E}_m \subset B(x,r)$, with $r<r_0$ to be determined shortly. Since $|\tilde{E}_m|$ is not necessarily 1, we proceed using \Cref{volume fixing lemma}. Let $y_1$, $y_2\in\pa^\ast (E^0)$ and $\eta>0$ be such that
\begin{equation*}
    \hn ( \pa^\ast (E^0) \cap B(y_i, \eta)) >0
\end{equation*}
for $i=1,2$ and
\begin{equation}\notag
    B(y_1, \eta) \cap B(y_2, \eta) = \emptyset.
\end{equation}
We apply the volume-fixing variations lemma with the choice of $A=B(y_i,\eta)$, yielding $\sigma_0$ and $C_0$ such that for any $|\sigma|<\sigma_0$ and $i=1,2$, there exists $G_m^i$ with $G_m^i \triangle E_m \subset\!\subset B(y_i,\eta)$ and
\begin{gather}\notag
    |G_m \cap B(y_i,\eta)| = |E_m\cap B(y_i,\eta)| + \sigma, \\ \label{volume fixing consequences}
    |G_m \triangle E_m| \leq C_0|\sigma|, \quad\textup{and}\quad|P(G_m;B(y_i,\eta))- P(E_m;B(y_i,\eta))| \leq C_0|\sigma|.
\end{gather}
Up to further decreasing $\sigma_0$, we may assume that
\begin{equation}\label{sigma0 bound}
    \max\{1,C_0\}\sigma_0 < \delta/2.
\end{equation}
Choose $r_0$ such that 
\begin{equation}\label{choice of r1}
\omega_n r_0^n < \sigma_0
\end{equation}
and for every $z\in \mathbb{R}^n$, $B(z,r_0)$ is disjoint from at least one of $B(y_i,\eta)$. Therefore, for at least one of $i=1,2$,
\begin{equation*}
    B(x,r) \cap B(y_i, \eta) = \emptyset;
\end{equation*}
let us assume without loss of generality that it is $y_1$. We introduce the sets
\begin{equation}\label{def of volumed fixed competitor}
    \overline{E}_m = (E_m \cap B(x,r)^c \cap B(y_1,\eta)^c) \cup ((G_m^1\cap B(y_1, \eta))\cup (\tilde{E}_m \cap B(x,r)),
\end{equation}
where $G_m^1$ is chosen according to \Cref{volume fixing lemma} with
\begin{equation}\label{sigmam bound}
\sigma_m:=|E_m\cap B(x,r)|-|\tilde{E}_m \cap B(x,r)|\in(-\omega_n r^n,\omega_n r^n),
\end{equation}
so that 
\begin{equation}\notag
    |G_m^1\cap B(y_1,\eta)| = |E_m\cap B(y_1,\eta)| +|E_m\cap B(x,r)|- |\tilde{E}_m \cap B(x,r)|.
\end{equation}
This ensures that
\begin{align*}\notag
|\overline{E}_m| &=   |E_m| - |E_m\cap B(x,r)| - |E_m \cap B(y_1,\eta)|  +|G_m^1\cap B(y_1, \eta)|+|\tilde{E}_m \cap B(x,r)|\\
&= |E_m| - |E_m\cap B(x,r)| - |E_m \cap B(y_1,\eta)|  + |E_m \cap B(y_1, \eta)| +|E_m\cap B(x,r)| \\ \notag
&\qquad- |\tilde{E}_m \cap B(x,r)|+ |\tilde{E}_m \cap B(x,r)|  \\
&=|E_m|\\ \notag
&=1.
\end{align*}
By the triangle inequality and the formula $\sigma_m=|E_m \cap B(x,r)|- |\tilde{E}_m \cap B(x,r)|$, the bound
\begin{equation}\label{symmetric difference bound for sigma}
    |\sigma_m| \leq |E_m \triangle \tilde{E}_m|
\end{equation}
holds as well. Furthermore, with the aid of \eqref{volume fixing consequences}--\eqref{choice of r1}, we may estimate $|\overline{E}_m \triangle E_m|$ by
\begin{align}\label{L1 estimate by symmetric difference}
    |\overline{E}_m \triangle E_m| &= |G_m^1 \triangle E_m| + |\tilde{E}_m \triangle E_m| \\ \notag
    &\leq C_0|\sigma_m| + \omega_n r^n \\ \notag
    &< \delta/2 + \delta/2.
\end{align}
The previous inequality implies that \eqref{symmetric difference inequality} holds for $E_m$ and $\overline{E}_m$, in which case
\begin{equation}\notag
     \left|\mathcal{W}_p(E_m) - \mathcal{W}_m(\overline{E}_m) \right| \leq C_{n,p} \alpha^{1-p}  |E_m \triangle \overline{E}_m|.
\end{equation}
Combining \eqref{L1 estimate by symmetric difference} and the fact that $|\sigma_m| \leq |E_m\triangle \tilde{E}_m|$, we have
\begin{equation}\label{symm diff bound}
     \left|\mathcal{W}_p(E_m) - \mathcal{W}_m(\overline{E}_m) \right| \leq C_{n,p} \alpha^{1-p} (C_0 + 1)|\tilde{E}_m \triangle E_m|.
\end{equation}
The last preliminary estimate before deriving \eqref{almost perimeter minimality of E_m} is a consequence of \eqref{volume fixing consequences} and \eqref{sigmam bound}:
\begin{align}\label{perimeter bound by symm diff}
| P(E_m;B(y_1,\eta)) - P(\overline{E}_m;B(y_1,\eta)) | \leq  C_0 |\sigma_m|  \leq C_0 |\tilde{E}_m \triangle E_m|.
\end{align}
Finally, since $|\overline{E}_m|=1$, we may test \eqref{almost minimality} with $\overline{E}_m$ and use \eqref{def of volumed fixed competitor}, \eqref{symm diff bound}, and \eqref{perimeter bound by symm diff} to obtain
\begin{align}\notag
    P(E_m)&\leq P(\overline{E}_m) + \lambda\mathcal{W}_p(\overline{E}_m) - \lambda\mathcal{W}_p(E_m)+\frac{1}{m} \\ \notag
    &=P(\tilde{E}_m;B(y_1,\eta)^c) +P(\overline{E}_m;B(y_1,\eta))-P(E_m;B(y_1,\eta)) + P(\tilde{E}_m;B(y_1,\eta)) \\ \notag
    &\qquad+ \lambda\mathcal{W}_p(\overline{E}_m) - \lambda\mathcal{W}_p(E_m)+\frac{1}{m}\\ \notag
    &\leq P(\tilde{E}_m) +C_0|\tilde{E}_m \triangle E_m|+ \lambda C_{n,p} \alpha^{1-p} (C_0 + 1)|\tilde{E}_m \triangle E_m|+\frac{1}{m}.
\end{align}
Taking $\Lambda:= C_0 + \lambda C_{n,p}\alpha^{1-p}(C_0+1)$, we have shown \eqref{almost perimeter minimality of E_m}.

\medskip

\noindent \textit{Step four:} Here we use \eqref{almost perimeter minimality of E_m} to prove that there exist $C_n$, $r_1>0$ such that \textit{any} positive measure set $E$ which is the $L^1_\loc$-limit of translates $E_m-y_m$ for a sequence $\{ y_m\}$ satisfies:
\begin{equation}\label{lower density estimate}
    |E \cap B(x,r) | \geq C_n r^n \quad \forall x \in \partial^\ast E,\, r<r_1.
\end{equation}
Since $|E|\leq 1$, such a lower density estimate implies that $\partial^* E$ and $E$ are bounded. For the proof of \eqref{lower density estimate}, to simplify the notation, assume that $y_m=0$ for all $m$.\par
We set
\begin{equation}\notag
    u_m(r) = |E_m \cap B(x,r)|,\quad u(r) = |E\cap B(x,r)|.
\end{equation}
The coarea formula implies that for almost every $r$,
\begin{equation}\notag
    u_m '(r) = \hn ( E_m\cap \partial B(x,r))\quad \textup{and}\quad u'(r) = \hn (E \cap \partial B(x,r)),
\end{equation}
while the $L^1_\loc$-convergence of $E_m$ to $E$ permits us to extract a subsequence such that
\begin{equation}\label{convergence of surface measures}
    u_m'(r) \to u'(r) \quad\textup{for almost every }r.
\end{equation}
Furthermore, except for a measure zero set of $r$ values which can be made independent of $m$, we have the identities
\begin{align}\label{perimeter of intersection with a ball}
    P(E_m ) &= P(E_m;B(x,r)) + P(E_m;\overline{B(x,r)}^c), \\ \notag
    P(E_m \cap B(x,r)) &= P(E_m;B(x,r)) + \hn (E_m \cap \partial B(x,r)),\\ \notag
    P(E_m \setminus B(x,r)) &= P(E_m;\overline{B(x,r)}^c) + \hn (E_m \cap \partial B(x,r)),
\end{align}
and similarly for $E$. Therefore, for almost every $r<r_1$, with $r_1\in (0,r_0)$ to be fixed shortly, testing \eqref{almost perimeter minimality of E_m} with $\tilde{E}_m= E_m \setminus B(x,r)$ yields
\begin{align}\notag
    P(E_m;& B(x,r)) + P(E_m ; \overline{B(x,r)}^c) \\ \notag
    &= P(E_m) \\ \notag
    &\leq P(E_m \setminus B(x,r)) + \Lambda |E_m \cap B(x,r)| + \frac{1}{m} \\ \notag
    &= P(E_m;\overline{B(x,r)}^c) + \hn(E_m \cap \partial B(x,r)) + \Lambda|E_m \cap B(x,r)| + \frac{1}{m}.
\end{align}
We add $\hn(E_m \cap \partial B(x,r))-P(E_m ; \overline{B(x,r)}^c)$ to both sides, arriving at  
\begin{multline}\notag
    P(E_m;B(x,r)) + \hn(E_m \cap \partial B(x,r)) \\ 
    \leq 2\hn(E_m \cap \partial B(x,r)) + \Lambda|E_m \cap B(x,r)| + \frac{1}{m}.
\end{multline}
The Euclidean isoperimetric inequality and \eqref{perimeter of intersection with a ball} imply that for almost every $r<r_1$,
\begin{align}\notag
    n \omega_n^{\frac{1}{n}} u_m(r)^{\frac{n-1}{n}} &= n \omega_n^{\frac{1}{n}} |E_m \cap B(x,r)|^{\frac{n-1}{n}} \\ \notag
    &\leq P(E_m \cap B(x,r)) \\ \notag
    &= P(E_m;B(x,r)) + \hn(E_m \cap \partial B(x,r)) \\ \notag
    &\leq 2\hn(E_m \cap \partial B(x,r)) + \Lambda|E_m \cap B(x,r)| + \frac{1}{m} \\ \label{differential inequality}
    &= 2u_m'(r) + \Lambda u_m + \frac{1}{m}.
\end{align}
With the goal of absorbing $\Lambda u_m$ into the left hand side, we note that
\begin{equation}\notag
    \Lambda u_m \leq \frac{n \omega_n^{\frac{1}{n}} u_m(r)^{\frac{n-1}{n}}}{2} \iff u_m \leq \left(\frac{n\omega_n^{1/n}}{2\Lambda} \right)^n.
\end{equation}
Therefore, choosing $r_1\in (0,r_0)$ small enough so that 
\begin{equation}\notag
u_m\leq \omega_n r_1^n \leq\left(\frac{n\omega_n^{1/n}}{2\Lambda} \right)^n,
\end{equation}
we have for almost every $r<r_1$
\begin{equation}\notag
    \Lambda u_m \leq \frac{n \omega_n^{\frac{1}{n}} u_m(r)^{\frac{n-1}{n}}}{2}.
\end{equation}
Plugging this into the differential inequality \eqref{differential inequality} and passing to the limit $m\to \infty$ using \eqref{convergence of surface measures}, we may write
\begin{align*}
    \frac{n \omega_n^{\frac{1}{n}} u(r)^{\frac{n-1}{n}}}{2} &= \lim_{m\to \infty} \frac{n \omega_n^{\frac{1}{n}} u_m(r)^{\frac{n-1}{n}}}{2} \\ 
    &\leq \lim_{m\to \infty} 2u_m'(r) + \frac{1}{m}\\
    &= 2 u'(r)
\end{align*}
for almost every $r<r_1$. The lower density estimate \eqref{lower density estimate} is achieved by {dividing by $u^{\frac{n-1}{n}}$} integrating this inequality.

\medskip

\noindent \textit{Step five:} In this step, we obtain finitely many, bounded, limiting sets $E^k$ and sequences $\{\xmk\}$ such that $E^k$ are $L^1_\loc$-limits of translates $E_m - x_m^k$ and satisfy
\begin{equation}\label{no volume loss}
    \sum_k |E^k|=1.
\end{equation}
To this end, apply the nucleation lemma again to $E_m$ with
\begin{equation}\notag
 \eps_0 = \min \left\{1, \frac{1+\inf \mathcal{G}}{2nc(n)}, C_nr_1^n \right\},    
\end{equation}
where $C_n$ is the dimensional constant from the previous step, to locate points $x_m^i$, $1\leq i \leq I(m)$ satisfying the conclusions of \Cref{nucleation lemma}. Here we include $C_n r_1^n$ in the definition of the constant $\eps_0$ as we would like to control the size of what is not contained in the balls obtained from the nucleation lemma. If the remainder is non-empty, its smallness will then lead to a contradiction with the lower density estimates.

The uniform bound \eqref{i bound} on $I(m)$ in terms of $P(E_m)$, $|E_m|$, and $\eps_0$ implies that by restricting to a further subsequence, we can find $I\in \mathbb{N}$ such that $I(m)=I$ for each $m$. After passing to a further subsequence, we may safely assume that
\begin{equation}\notag
    \lim_{m \to \infty} |x_m^i - x_m^j| =: d_{ij}
\end{equation}
exists for each pair $(i,j) \in I \times I$, with infinity as a possible limit, too. Next, we define equivalence classes of $\{1,\dots, I\}$ based on the relation
\begin{equation}\notag
    i \equiv j \iff d_{ij} < \infty.
\end{equation}
Let $K \leq I$ be the number of these equivalence classes, which partition $\{1,\dots, I\}$. For each $1\leq k \leq K$ and $m\in \mathbb{N}$, let $x_m^k:=x_m^{i(k,m)}$ be a point from the family of points corresponding to $E_m$ such that $i(k,m)$ is a representative of the $k$-th equivalence class. Recall that due to \eqref{lower bound on i ball} and \eqref{perimeter upper bound}, $E_m - \xmk$ satisfies
\begin{equation}\notag
    | (E_m - \xmk) \cap B(0,1) | \geq  \left(c(n) \frac{\eps_0}{P(E_m)} \right)^n\geq \left(c(n) \frac{\eps_0}{1+\inf \mathcal{G}} \right)^n.
\end{equation}
We can therefore find non-trivial sets of finite perimeter $E^k$ such that
\begin{equation}\label{convergence to E^k}
    E_m - \xmk \overset{\loc}{\to} E^k.
\end{equation}
Since the previous step implies that each $E^k$ is bounded, there exists $R_0$ such that
\begin{equation}\label{ball containing E^k}
    E^k \subset\!\subset B(0,R_0)
\end{equation}
for each $1\leq k \leq K$. We may also take $R_0$ to be large enough that 
\begin{equation}\label{large ball contains equivalence classes}
    \bigcup_{i\in\{1,\dots, I\}\colon i\equiv k} B(x_m^i,2) \subset  B(\xmk, R_0) 
\end{equation}
for all $m$; in other words $B(\xmk,R_0)$ contains all the balls at the $m$-th stage with indices in the same equivalence class as $k$.
\par
It remains to show that
\begin{equation}\notag
    \sum_{k=1}^K |E^k| =1.
\end{equation}
We first show that $\sum_{k=1}^K |E^k|\leq 1$. If this were not the case, then 
\begin{equation}\label{vol of E_k on ball}
    \sum_{k=1}^K |E^k \cap B(0,R_0)| > 1.
\end{equation}
Now for large $m$, the sets
\begin{equation*}
    E_m \cap B(\xmk, R_0)
\end{equation*}
are pairwise disjoint since $|\xmk - x_m^{k'}|\to \infty$ if $k\neq k'$. By \eqref{convergence to E^k} and \eqref{vol of E_k on ball}, it follows that
\begin{equation*}
    \sum_{k=1}^K |E_m \cap B(\xmk,R_0)|>\frac{1}{2}\left(1+\sum_{k=1}^K |E^k \cap B(0,R_0)|\right)
\end{equation*}
for large $m$, which is impossible since $ |E_m| = 1$. So
\begin{equation*}
    \sum_{k=1}^K |E^k|  \leq 1.
\end{equation*}

Assume now for a contradiction that
\begin{equation}\notag
    \sum_{k=1}^K |E^k| = 1-\delta
\end{equation}
for some $\delta>0$. Since $E^k\subset\!\subset B(0,R_0+2)$ and $E_m \cap B(\xmk, R_0+2)$ are disjoint for large enough $m$, it must then be the case that
\begin{equation}\notag
    \left|E_m \setminus \left(\bigcup_k B(\xmk,R_0+2)\right)\right| \geq \frac{\delta}{2}
\end{equation}
for large enough $m$. At the same time, the nucleation lemma at the beginning of this step with $\eps_0\leq C_n r_1^n$ gave
\begin{equation}\notag
    \left|E_m \setminus \bigcup_{1\leq i\leq I} B(x_m^i,2) \right| < \eps_0 \leq C_n r_1^n.
\end{equation}
Together with the assumption \eqref{large ball contains equivalence classes} that $\bigcup_i B(x_m^i,2) \subset \bigcup_k B(\xmk,R_0)$, this yields
\begin{equation}\label{small measure at mth stage}
  \frac{\delta}{2}\leq  \left|E_m \setminus \left(\bigcup_k B(\xmk,R_0+2)\right)\right|  < C_n r_1^n.
\end{equation}
Applying the nucleation lemma a final time to the sets $ E_m \setminus\left(\bigcup_k B(\xmk,R_0+2)\right)$, we obtain finitely many points $y_m^j$ fulfilling the conclusions of \Cref{nucleation lemma}. We claim that it must be the case that
\begin{equation}\label{divergence from previous balls}
    |y_m^j - \xmk|\to \infty.
\end{equation}
If $\limsup_{m \to \infty} |y_m^j - \xmk|<\infty$, then the uniform bound from below on $|B(y_m^j,1) \cap E_m|$ and the fact that $y_m^j \notin B(\xmk, R_0+1)$ would imply that $E^k\cap B(0,R_0)^c\neq \emptyset$. However, this contradicts \eqref{ball containing E^k}. Next, by the compactness for sets of finite perimeter and the fourth step, we may find a measurable set $E$ and $R_1>0$ such that $E \subset\!\subset B(0,R_1)$, $E_m - y_m^1  \overset{\loc}{\to} E$, and
\begin{equation}\label{measure too large}
|E| \geq C_n r_1^n .
\end{equation}
Since $E$ is compactly supported, $|E_m \cap B( y_m^1,R_1)| \to |E| $. But \eqref{divergence from previous balls} implies that $B(y_m^1, R_1) \subset \left(\bigcup_k B(\xmk,R_0+2)\right)^c$ for large $m$, and hence
\begin{align*}
    |E| &= \lim_{m\to \infty}| E_m \cap B(y_m^1, R_1)| \\
    & \leq \limsup_{m\to \infty}\left|E_m \setminus \left(\bigcup_k B(\xmk,R_0+2)\right)\right| \\
    &\hspace{-.24cm}\xupref{small measure at mth stage}{<}
    C_n r_1^n.
\end{align*}
This upper bound is at odds with the lower bound \eqref{measure too large}, so we have derived a contradiction. Thus $\sum_{k=1}^K |E^k| = 1$.

\medskip

\noindent \textit{Step six:} At last we can prove \Cref{existence theorem}. Let us choose any $K$ points $z_1,\dots, z_K\in \mathbb{R}^n$ such that $B(z_k, R_0+C_n)$ are pairwise disjoint, where $C_n$ is the dimensional constant from \Cref{Buttazzo existence theorem}(ii). We claim that
\begin{equation}\notag
    \bigcup_{k=1}^K E^k + z_k
\end{equation}
is a minimizer. The choice of radius $R_0+C_n$ and \Cref{additivity of script W} ensure that the sets $F^k$ defined by
\begin{equation*}
    \mathcal{W}_p(E^k+z_k) = W_p(E^k+z_k, F^k)
\end{equation*}
are pairwise disjoint and 
\begin{equation*}
    \mathcal{W}_p^p\left( \bigcup_{k=1}^K E^k + z_k\right) = \sum_{k=1}^K \mathcal{W}_p^p(E^k).
\end{equation*}
Appealing to the continuity result \Cref{L1 control of script W} gives
\begin{align}\notag
   \lambda \mathcal{W}_p^p\left(\bigcup_{k=1}^K E^k + z_k\right)&=\lambda \sum_{k=1}^K \mathcal{W}^p_p(E^k) \\ \notag
    &=\lambda \lim_{m\to \infty} \sum_{k=1}^K \mathcal{W}^p_p((E_m-\xmk) \cap B(0, R_0+C_n))\\ \label{Wp lsc}
    &\leq \lambda\liminf_{m\to \infty} \mathcal{W}^p_p(E_m),
\end{align}
where the last inequality depends on \Cref{additivity of script W}, the additivity of $\mathcal{W}_p^p$ (which applies since the distance between the $x_m^k$'s goes to infinity as $m\to \infty$). Next, the inequality
\begin{align}\notag
    \sum_{k=1}^K P(E^k+z_k) &= \sum_{k=1}^K P(E^k; B(0,R_0))\\ \notag
    &\leq \liminf_{m\to \infty} \sum_{k=1}^K P(E_m-\xmk ;B(0,R_0))\\ \label{perimeter lsc}
    &\leq \liminf_{m\to\infty} P(E_m)
\end{align}
is immediate from the lower-semicontinuity of perimeter under $L^1$-convergence and the pairwise disjointness again. Summing \eqref{Wp lsc} and \eqref{perimeter lsc} finishes the proof, since $E_m$ is a minimizing sequence and $\left|\bigcup_{k=1}^K E^k + z_k\right|=1$.
\end{proof}

\bigskip

{ As a byproduct of our existence proof we obtain that the set $E$ in a minimizing pair $(E,F)$ is a quasiminimizer of the perimeter in the following sense; hence, enjoys some regularity properties.}

\begin{corollary}
For any minimizing pair $(E,F)$ to \eqref{the functional}, the set $E$ is a $(\Lambda,r_0)$-perimeter minimizer in $\mathbb{R}^n$. That is, there exists $0\leq \Lambda<\infty$ and $r_0>0$ such that 
\begin{equation}\notag
    P(E) \leq P(\tilde{E}) + \Lambda|E \triangle \tilde{E}|\qquad \textup{if } E \triangle \tilde{E} \subset\!\subset B(x,r),\, 0<r<r_0.
\end{equation}
\end{corollary}
\begin{proof}
The analogous inequality for the elements of the minimizing sequence $\{E_m\}$ was derived in \eqref{almost perimeter minimality of E_m} with an added factor of $1/m$, and the same proof applies to the minimizer $E$. 
\end{proof}

\smallskip

\begin{remark}[Regularity of minimizers]
The classical theory of $(\Lambda,r_0)$-perimeter minimality implies that $\partial^* E \in C^{1,\gamma}$ for any $\gamma\in(0,1/2)$ and the Hausdorff dimension of $\partial E \setminus \partial^* E$ is at most $n-8$ (see e.g. \cite[Theorem 26.3]{Maggi}). {This regularity was also observed in \cite[Theorem 4.6]{Butt}.} Also, by \cite[Theorem 3.13]{Butt}, $F$ is a set of finite perimeter.
\end{remark}

\medskip

\begin{remark}
\label{r.frank}
Alternatively, one could attempt to demonstrate the existence of minimizers using the framework developed in \cite{FL15}. This would require proving that the binding inequality
    \[
        e(M) < e(M^\prime) + e(M-M^\prime)
    \]
holds for all $0<M^\prime<M$, where $e(M) = \inf \big\{ P(E) + \mathcal{W}_p(E) \colon |E|=M \big\}$.
\end{remark}

\bigskip

\appendix
\section{Proof of \Cref{volume fixing lemma}}
The argument is a straightforward modification of the case where there is one set $E$ \cite[Lemma 17.21]{Maggi}, as opposed to a sequence.
\begin{proof}[Proof of \Cref{volume fixing lemma}]
Since $\mathcal{H}^{n-1}(\partial^\ast E\cap A)>0$, we can find $T\in C_c^\infty(A;\mathbb{R}^n)$ with
\begin{equation}\notag
    \gamma := \int_E \Div T \,dx = \int_{\partial^\ast E} T \cdot \nu_{E} \, d\mathcal{H}^{n-1}>0.
\end{equation}
By the $L^1_\loc$-convergence of $E_m$ to $E$, for $m$ large enough, we have
\begin{equation}\label{bounds on boundary of Em}
    \frac{\gamma}{2} < \int_{E_m} \Div T \,dx= \int_{\partial^\ast E_m} T \cdot \nu_{E_m} \, d\mathcal{H}^{n-1} < 2\gamma.
\end{equation}
Let $\varphi_t(x)\colon \RR^n \times (-\delta,\delta)\to \RR^n$ be a one parameter family of diffeomorphisms with initial velocity $T$. By the first variation formulae for perimeter and volume (see e.g. \cite[Chapter 17]{Maggi}), there exists $\delta_0>0$ such that for all $|t| \leq \delta_0$,
\begin{gather}\label{first variation of perimeter}
    \left|P(\varphi_t(E_m);A) - P(E_m;A) \right| \leq 2|t|  P(E_m; A) \| \nabla T\|_{L^\infty},\\ \label{first variation of volume}
    |\varphi_t(E_m)\cap A| = |E_m\cap A| + t\int_{\partial^\ast E_m} T \cdot \nu_{E_m} \,d\mathcal{H}^{n-1}+O(t^2),
\end{gather}
where the decay rate in $t$ in the second equality depends on $T$ and is thus uniform in $m$. Also, by \eqref{bounds on boundary of Em} and \eqref{first variation of volume}, $|\varphi_t(E_m)\cap A|$ is strictly increasing on $[-\delta_0,\delta_0]$ with 
\begin{equation}\label{slope bound}
    \big| \, |\varphi_t(E_m)\cap A| - |\varphi_{t'}(E_m)\cap A| \, \big| \geq \frac{\gamma}{4}|t - t'|
\end{equation}
(after decreasing $\delta_0$ if necessary). Therefore, we have the inclusion
\begin{equation}\notag
    \left(-\frac{\delta_0 \gamma}{4},\frac{\delta_0 \gamma}{4}\right) \subset \big\{|\varphi_t(E_m)\cap A|-|E_m \cap A|\colon |t| \leq \delta_0 \big\}.
\end{equation}
So for all $|\sigma|<\sigma_0:=\delta_0\gamma/4$, there exists $t_m=t_m(\sigma)\in (-\delta_0,\delta_0)$ such that 
\begin{equation}\label{volume of Gm}
    |\varphi_t(E_m)\cap A| = |E_m\cap A| + \sigma.
\end{equation}
By \eqref{slope bound}, it must be the case that
\begin{equation}\label{tm bound}
    |t_m| < \frac{4\sigma}{\gamma}.
\end{equation}
Then defining $G_m = \varphi_{t_m}(E_m)$, it follows from \eqref{volume of Gm} and \eqref{first variation of perimeter}, \eqref{tm bound} that
\begin{align}\notag
    |G_m \cap A| &= |E_m\cap A| + \sigma, \quad|P(G_m;A)- P(E_m;A)| \leq C_0|\sigma|,
\end{align}
where $C_0$ depends on $M=\sup P(E_m;A)$, $A$, and $E$. The estimate
\begin{equation}\notag
|G_m \triangle E_m| \leq C_0|\sigma|
\end{equation}
can be found in \cite[Lemma 17.9]{Maggi} in the form 
\begin{equation}\notag
    | \varphi_{t_m}(E_m) \triangle E_m| \leq C(T)\, |t_m|\,P(E_m;A).
\end{equation}
Hence, the result is established.
\end{proof}

\section*{Acknowledgements}
We would like to thank Rupert L. Frank for his valuable comments. MN's research is supported by the NSF grant RTG-DMS 1840314. IT's research is partially supported by the Simons Collaboration Grant for Mathematicians No. 851065. RV acknowledges partial support from the AMS-Simons Travel Grant. 

\bibliographystyle{IEEEtranSA}
\bibliography{ref}

\end{document}